\documentclass[11pt]{article}
\usepackage{amsmath, amssymb, amsfonts, amsthm}
\usepackage{enumitem}
\usepackage{multicol}
\usepackage{hyperref}
\allowdisplaybreaks
\def \R {{\mathbb{R}}}
\def \Q {{\mathbb{Q}}}
\def \N {{\mathbb{N}}}
\def \Z {{\mathbb{Z}}}

\def \S {{\mathcal{S}}}

\def \ep {{\epsilon}}
\def \ze {{\zeta}}

\newcommand{\sgn}{\operatorname{sgn}}

\newtheorem*{theorem*}{Theorem}
\newtheorem{theorem}{Theorem}
\newtheorem{cor}[theorem]{Corollary}
\newtheorem{conj}[theorem]{Conjecture}
\newtheorem{lemma}[theorem]{Lemma}

\newtheorem{rem}[theorem]{Remark}

\newtheorem{pro}[theorem]{Proposition}

\title{An irreducible class of polynomials over integers}
\author {Biswajit Koley,  A.Satyanarayana Reddy\footnote{The research of this author is supported by Matrics MTR/2019/001206 of SERB, India.} \\
Department of 
Mathematics, Shiv Nadar 
University, India-201314\\ (e-mail: 
bk140@snu.edu.in, satyanarayana.reddy@snu.edu.in).
  }
\date{}
\begin{document}
\maketitle
\begin{abstract}
In this article, we consider polynomials of the form $f(x)=a_0+a_{n_1}x^{n_1}+a_{n_2}x^{n_2}+\cdots+a_{n_r}x^{n_r}\in \Z[x],$
where $|a_0|\ge |a_{n_1}|+\dots+|a_{n_r}|,$ $|a_0|$ is a prime power and $|a_0|\nmid |a_{n_1}a_{n_r}|$. We will show that under the strict inequality  these polynomials are irreducible for certain values of $n_1$. In the case of equality, apart from its cyclotomic factors, they have exactly one irreducible non-reciprocal factor.  
\end{abstract}
{\bf{AMS(2010)}}: 11R09, 12D05, 12D10.\\

\section{Introduction}\label{sec:intro}
The question of finding an irreducibility criterion for polynomials depending upon its coefficients has been studied extensively. One of the well-known criteria is  Eisenstein's criterion~\cite{eisen}, which demands prime decomposition of the coefficients of a given polynomial. Another famous criterion known as  Perron's criterion~\cite{perron}, does not require prime decomposition of the coefficients:
\begin{theorem}[Perron~\cite{perron}]
 Let $f(x)=x^n+a_{n-1}x^{n-1}+\dots+a_1x+a_0\in \Z[x]$ be a monic polynomial with $a_0\ne 0.$
 If $|a_{n-1}|>1+|a_{n-2}|+|a_{n-3}|+\dots+|a_1|+|a_0|,$ then $f(x)$ is irreducible over $\Z.$
\end{theorem}

Finding an irreducibility criterion similar to Perron is of great interest to mathematicians. One of the works  in this direction is due to Panitopol and Stef\"anescu \cite{LP}. They studied polynomials with integer coefficients having a constant term as a prime number and  proved the following.
\begin{theorem}\label{th1}
If $p$ is a prime and $p>|a_1|+\cdots+|a_n|,$ then $a_nx^n+\cdots+a_1x\pm p$ is irreducible. 
\end{theorem}

 A.I. Bonciocat and N.C. Bonciocat  extended this work  in~\cite{bonciocat1} and~\cite{bonciocat2} to prime powers. Before stating their results, we define $\tilde{g}(x)=x^{\deg(g)}g(x^{-1})$ as the {\em reciprocal polynomial} of the polynomial $g(x).$ It is easy to see that both $g(x)$ and $\tilde{g}(x)$ reducible or irreducible together provided $g(0)\ne 0$. A polynomial $f(x)$ is said to be {\em reciprocal} if $f(x)=\pm \tilde{f}(x),$ otherwise it is called {\em non-reciprocal}. With this notation and observation we express the results of A.I. Bonciocat and N.C. Bonciocat  in terms of $\tilde{f}(x)$.  In~\cite{bonciocat1} it is shown that if $p$ is a prime number, $p\nmid a_{2}a_0$, $p^u>|a_0a_{2}|p^{3e} +\sum_{i=3}^n |a_{0}^{i-1}a_{i}|p^{ie}$ and $u\not\equiv e\pmod{2}, u\ge 1, e\ge 0,$ then $a_nx^n+\cdots+a_{2}p^ex^2+a_0p^u$ is irreducible. In~\cite{bonciocat2} they proved a similar result for $a_{1}\ne 0$ instead of $a_{2},$ that is, if  $p\nmid a_0a_{1}$, $p^u>|a_{1}|p^{2e}+\sum_{i=2}^n|a_{0}^{i-1}a_{i}|p^{ie}, u\ge 1, e\ge 0,$ then $a_nx^n+\cdots+a_{1}p^ex+a_0p^u$ is irreducible.

Note that if $e=0, a_0=1$, then both of these conditions are the same as that of Theorem~\ref{th1}.  A similar study of the irreducibility of polynomials with constant term divisible by a prime or prime power can be found in \cite{bksr}, \cite{jankauskas}, \cite{jonassen}, \cite{lipka}, \cite{weisner}. For example, Jonassen\cite{jonassen} gave a complete factorization of trinomials of the form $x^n\pm x^m\pm 4$. He proved that they are irreducible except for six distinct families of polynomials.  Weisner~\cite{weisner} proved that if $p$ is a prime number and $n\ge 2, m\ge 1,$ then $x^n\pm x\pm p^m$ is irreducible whenever $p^m>2$. The authors~\cite{bksr} have shown that apart from cyclotomic factors, $x^n\pm x\pm 2$ has exactly one non-reciprocal irreducible factor.

Suppose $n_r>n_{r-1}>\cdots>n_1>0$ and $p$ is a prime number. Let 
$$\S_{n_1}=\{a_{n_r}x^{n_r}+a_{n_{r-1}}x^{n_{r-1}}+\cdots+a_{n_1}x^{n_1}+p^u\ep|u\ge 2, \ep=\pm 1, p\nmid a_{n_1}a_{n_r}\; \mbox{and} \;a_{n_i}\ne 0\}$$ and   $\S_{n_1}^\prime=\{\, f\in \S_{n_1}\mid u \not\equiv 0\pmod{n_1}\,\}$. It is clear that for $n_1=1$, $\S_{n_1}'=\emptyset$. With these notations, the results of \cite{bonciocat1} and \cite{bonciocat2} can be combined as: {\em $f\in \S_1\cup \S_2'$ is irreducible if $p^u>|a_{n_1}|+\cdots+|a_{n_r}|$ }. The main result in  this article is the following.

\begin{theorem}\label{mainthm}
Let $f(x)=a_{n_r}x^{n_r}+a_{n_{r-1}}x^{n_{r-1}}+\cdots+a_{n_1}x^{n_1}+p^u\ep\in \S_1\cup \S_2'\cup \S_3'$ and $p^u>|a_{n_1}|+\cdots+|a_{n_r}|$.  Then $f(x)$ is irreducible. 
\end{theorem}

Above result need not be  true if $f(x)\in \S_2\cup \S_3.$ For example, $x^4+4\ep x^3+3^3\in \S_3\setminus \S_3'$ and  
$$x^4+4\ep x^3+3^3=(x+3\ep)^2(x^2-2\ep x+3),\;\; \mbox{where}\; \ep=\pm 1.$$ In \cite{bonciocat1}, it is given that $
f(x)=x^4+(2^{k+1}-1)x^2+2^{2k}=(x^2+x+2^k)(x^2-x+2^k)\in \S_2\setminus \S_2',$ for every $k\ge 1$.  Another example, $x^7+x^5+x^3+2^3=(x^3-x^2-x+2)(x^4+x^3+3x^2+2x+4)$  is the problem $007:14$ stated at West Coast Number Theory conference in 2007 by  Walsh~\cite{myerson}. The  example 
\begin{equation*}
x^{12}+x^8+x^4-16=(x^3-x^2-x+2)(x^3+x^2-x-2)(x^6+3x^4+5x^2+4)
\end{equation*}
is collected from~\cite{jankauskas}.
 
The following examples  illustrate the necessity of the condition $p\nmid a_{n_r}a_{n_1}$
in the definition of $\S_{n_1}:$
\begin{align}
 & x^3-x^2-10x+16=(x-2)(x^2+x-8);\notag\\ 
 & 2x^3-3x^2-27=(x-3)(2x^2+3x+9);\notag\\
 & 3x^6+x^5-3x^3-81=(x^2-3)(3x^4+x^3+9x^2+27);\notag\\
 & x^8+2x^6+6x^4-81=(x^2+3)(x^6-x^4+9x^2-27).\notag
\end{align}
The last example shows that it is not possible to drop the condition $p|a_{n_1}$ even for larger values of $n_1$. The condition on the coefficients  given in Theorem \ref{mainthm} enforces all the roots of $f$ to lie outside the unit circle. More generally,

\begin{rem}\label{rem:non-reciprocal}
 Let $f(x)=a_{n_r}x^{n_r}+a_{n_{r-1}}x^{n_{r-1}}+\cdots+a_{n_1}x^{n_1}+a_0$ be any polynomial of degree $n_r$ and $|a_0|>|a_{n_1}|+\cdots+|a_{n_r}|.$  Then every root of $f(x)$
 lies outside the unit circle.
 \begin{proof}
Let $z$ be a root of $f(x)$ with $|z|\le 1$. Then $f(z)=0$ and taking modulus on both sides of 
\begin{equation*}
-a_0=a_{n_1}z^{n_1}+\cdots+a_{n_r}z^{n_r},
\end{equation*}
we get $|a_0|\le |a_{n_1}|+\cdots+|a_{n_r}|$ which contradicts the hypothesis. Therefore, all the roots of $f(x)$ lies in the region $|z|>1$.  
 \end{proof}
\end{rem}
Recall that if $z\ne 0$ is a root of a reciprocal polynomial, then so is  $\frac{1}{z}.$  In other words, every reciprocal polynomial contains a root that lies inside or on the unit circle. Hence, if $f(x)$ satisfies the hypothesis of Remark~\ref{rem:non-reciprocal}, then every factor of $f(x)$ is non-reciprocal. A natural question: is it possible to find number of factors of $f(x)?$

 The remark of  Schinzel given by Jankauskas in~\cite{jankauskas} states that there are at most $\Omega(k)$ irreducible non-reciprocal factors for polynomials of the form $x^n+x^m+x^r+k, k\in \N$, where $\Omega(k)$ denotes the total number of prime factors of $k$ with repetitions. Jankauskas~\cite{jankauskas} gave the following example
\begin{equation*}
x^{12}+x^8+x^4+52=(x^2-2x+2)(x^2+2x+2)(x^8-3x^4+13)
\end{equation*}
to establish the sharpness of the above remark. However, for the present family of polynomials, the number of irreducible factors is usually much less than that of $\Omega(p^u)=u$. We apply the method followed by Ljunggren\cite{lju} to study the behavior of the factors of $f(x)\in \S_{n_1}$ and we will show that 
\begin{cor}\label{cor:n_1=2}
Suppose $f(x)=a_{n_r}x^{n_r}+a_{n_{r-1}}x^{n_{r-1}}+\cdots+a_{n_1}x^{n_1}+p^u\ep\in \S_{n_1}$ is reducible, where $n_1\in \{2,3\}$ and $p^u>|a_{n_1}|+\cdots+|a_{n_r}|$. Then $f(x)$ has at most $n_1$   non-reciprocal irreducible factors. 
\end{cor}

Later we consider the equality condition $p^u=|a_{n_1}|+\cdots+|a_{n_r}|$. The authors have already considered the case $u=1$ and $p=|a_{n_1}|+\cdots+|a_{n_r}|$ in~\cite{bksr}. Here we will establish similar results for $u\ge 2$. 

\begin{theorem}\label{cyclothm}
Let $f(x)=a_{n_r}x^{n_r}+a_{n_{r-1}}x^{n_{r-1}}+\cdots+a_{n_1}x^{n_1}+p^u\ep\in \S_1\cup \S_2'\cup \S_3'$ be  reducible and $|a_{n_r}|+|a_{n_{r-1}}|+\cdots+|a_{n_1}|=p^u$. Then 
$f(x)=f_c(x)f_n(x)$, where $f_n(x)$ is the irreducible non-reciprocal factor of $f(x)$ and 
$f_c(x)=\gcd(x^{n_1}+\sgn(a_{n_1}\ep),\ldots,x^{n_r}+\sgn(a_{n_r}\ep) ),$ $\sgn(x)$ being  the sign of $x\in \R.$ 
\end{theorem}

  The following example does not satisfies the hypothesis of Theorem~\ref{cyclothm} and it has more than one non-reciprocal factor. However, the cyclotomic factor arises from the expression of $f_c(x)$ given in Theorem \ref{cyclothm},
\begin{equation*}
4x^6+5x^2+9=(x^2+1)(2x^2-4x+3)(2x^2+4x+3).
\end{equation*}

There are polynomials for which $p|a_{n_1}a_{n_r}$ and they may or may not be of the form $f_c(x)f_n(x)$. For example, 

\begin{enumerate}
 \item[$n_1=1$:] \begin{align}
& 3x^4+11x^2+2x+16=(x^2+x+2)(3x^2-3x+8);\notag\\
& 9x^5+5x^3+2x+16=(x+1)(9x^4-9x^3+14x^2-14x+16),\label{eq:eq3}
\end{align}

 \item[$n_1=2$:]  \begin{align}
& 3x^8+2x^6+9x^4+2x^2+16=(x^4-x^2+2)(3x^4+5x^2+8);\notag\\
& 9x^{10}+5x^6+2x^2+16=(x^2+1)(9x^8-9x^6+14x^4-14x^2+16),\label{eq:eq2}
\end{align}

 \item[$n_1=3$:] 
\begin{align}
& 3x^{12}+11x^6+2x^3+16=(x^6+x^3+2)(3x^6-3x^3+8);\notag\\
& 5x^{15}+9x^9+2x^3+16=(x+1)(x^2-x+1)(5x^{12}-5x^9+14x^6-14x^3+16),\label{eq:eq1}
\end{align}
\end{enumerate}

Equations \eqref{eq:eq3},\eqref{eq:eq2},\eqref{eq:eq1} shows that the form of $f_c(x)$ in the Theorem~\ref{cyclothm}       is same when $n_1\le 3$ even though they do not belong to $\S_{n_1}$. This motivates us to show that the second part of Theorem~\ref{cyclothm} is true even for a larger class of polynomials.

\begin{pro}\label{pro:fc(x)}
Let $f(x)=a_{n_r}x^{n_r}+\cdots+a_{n_1}x^{n_1}+a_0\in \Z[x]$ be a polynomial with $|a_0|= |a_{n_1}|+\cdots+|a_{n_r}|$ and $f_c(x)= \gcd(x^{n_r}+\sgn(a_0a_{n_r}), x^{n_{r-1}}+\sgn(a_0a_{n_{r-1}}), \ldots, x^{n_1}+\sgn(a_0a_{n_1}))$. If $f(x)$ has a cyclotomic factor, then $f_c(x)|f(x)$ and  $f_c(x)$ is the product of all cyclotomic factors of $f(x).$ 
\end{pro}

 The polynomial $f(x)=x^4+x^2-8$ is irreducible, but $f_c(x)=\gcd(x^4-1,x^2-1)=x^2-1$ does not divided $f(x).$  Let $g(x)=x^4+(k+1)x^3+x^2-k,$ where $k\ge 1.$ Then $g(x)=\Phi_3(x)h(x)$ for some $h(x)\in \Z[x]$ where as $g_c(x)=\gcd(x^4-1, x^3-1, x^2-1)=x-1$  and $g(1)\ne 0$. Thus, Proposition~\ref{pro:fc(x)} is no longer true apart from the equality condition on the coefficients. One can conclude that for any $f(x)=a_{n_r}x^{n_r}+\cdots+a_{n_1}x^{n_1}+a_0\in \Z[x],$ the followings hold: 
 \begin{enumerate}
  \item if $|a_0|>|a_{n_1}|+\cdots+|a_{n_r}|$, then from the Remark~\ref{rem:non-reciprocal}, $f(x)$ is not divisible by a   cyclotomic polynomial.
  \item if $|a_0|=|a_{n_1}|+\cdots+|a_{n_r}|$, then from the Proposition~\ref{pro:fc(x)},  $f(x)$ has cyclotomic factors if and only if $f_c(x)\ne 1.$
 \end{enumerate}

Let $n$ be a positive integer. We denote $e(n)$ as the largest even part of $n$, that is if $n=2^an_1$ with $n_1$ being odd, then $e(n)=2^a$.  Under some special restrictions on the exponents of $x$ in  $f(x)$, Theorem~\ref{cyclothm} provides various useful irreducibility criterion  for polynomials of this nature. For example, 
\begin{cor}\label{cor:pos}
Suppose $f(x)=a_{n_r}x^{n_r}+a_{n_{r-1}}x^{n_{r-1}}+\cdots+a_{n_1}x^{n_1}+p^u\ep\in (\S_1\cup \S_2'\cup \S_3')\cap \Z_+[x]$ is a polynomial and $a_{n_1}+a_{n_2}+\dots+a_{n_{r-1}}+a_{n_r}=p^u$.  Then $f(x)$ is irreducible if and only if there exist distinct $i,j$ such that $e(n_i)\ne e(n_j).$  
\end{cor}

 Few applications of these results in the case of trinomials are shown in section~\ref{sec:app}.

\section{Proofs}\label{sec:proofs}
Suppose $n,m$ are two positive integers. It is known that $(x^n-1,x^m-1)=x^{(n,m)}-1$. We will use the following lemma later in the paper to draw several consequences of Theorem \ref{mainthm} and Theorem \ref{cyclothm}. See \cite{bksr} for the detailed proof.
\begin{lemma}\label{lem:basic_cyclo}
Suppose $n,m$ are two positive integers. Then $$ (x^n+1, x^m+1)=\begin{cases}
x^{(n,m)}+1 &\mbox{ if $e(m)=e(n);$}\\
1 &\mbox{ otherwise,}
\end{cases}$$ and 
$$(x^n+1, x^m-1)=\begin{cases} x^{(n,m/2)}+1 &\mbox{ if $e(m)\ge 2e(n);$}\\
1 &\mbox{ otherwise.}
\end{cases}$$
\end{lemma}

 Since it is known that Theorem~\ref{mainthm} is true when $f\in \S_1\cup \S_2'$, we prioritize the irreducibility of $f\in \S_3'$. Theorem~\ref{mainlem} will  provide an alternate proof for the irreducibility of $f\in \S_1\cup \S_2'$ by the approach followed in the proof of Lemma~\ref{lem1}.
 
 \begin{rem} 
  If $f\in \S_1\cup \S_2\cup\S_3$, then either $f(0)=p^u$ or 
  $f(0)=-p^u.$ Since irreducible factors of $f(x)$ and $-f(x)$ are same upto sign, without loss of generality we will assume at least one of the irreducible factors of $f(x)$ has a positive constant term.
 \end{rem}

\begin{lemma}\label{lem1}
Let $f(x)=a_{n_r}x^{n_r}+a_{n_{r-1}}x^{n_{r-1}}+\cdots+a_{n_1}x^{n_1}+p^u\ep\in \S_3'$ be reducible. Then the constant term of the  one of the irreducible factors of $f(x)$ is $|f(0)|.$
\end{lemma}

\begin{proof}
Suppose $f(x)=f_1(x)f_2(x)$ is a non trivial factorization of $f(x)$ with $\deg(f_1)=s$.  Let $g(x)=f_1(x)\tilde{f}_2(x)=\sum\limits_{i=0}^{n_r} b_ix^i.$
Then $\tilde{g}(x)=\sum\limits_{i=0}^n b_{n_r-i}x^i$. Since $g(x)\tilde{g}(x)=f(x)\tilde{f}(x)$, comparing the leading coefficient and the coefficient of $x^{n_r}$, we get
\begin{equation*}
b_0b_{n_r}=p^ua_{n_r}\ep; \qquad \sum\limits_{i=0}^{n_r} b_i^2=p^{2u}+\sum\limits_{i=1}^{r}a_{n_i}^2,
\end{equation*}
respectively. Let $b_0=p^{\alpha}d$ and $b_{n_r}=p^{u-\alpha}d_1,$ where $dd_1=a_{n_r}\ep$ and $\alpha\ge 0$.  Then the  equation $\sum\limits_{i=0}^{n_r} b_i^2=p^{2u}+\sum\limits_{i=1}^{r}a_{n_i}^2$ can be written as
\begin{equation*}
\sum\limits_{i=1}^{n_r-1}b_i^2=p^{2u}-p^{2\alpha}d^2-p^{2(u-\alpha)}d_1^2+\sum\limits_{i=1}^{r}a_{n_i}^2.
\end{equation*}
Suppose that $b_i$ is nonzero whenever $i\in\{0,j_1,j_2,\ldots,j_t,n_r\},$ where    $0<j_t<j_{t-1}<\cdots<j_1<n_r.$ Then $g(x)=b_{n_r}x^{n_r}+b_{j_1}x^{j_1}+\cdots+b_{j_t}x^{j_t}+b_0$ and 
\begin{equation}\label{eq1}
g(x)\tilde{g}(x)=p^ua_{n_r}\ep x^{2n_r}+b_{n_r}b_{j_t}x^{2n_r-j_t}+b_0b_{j_1}x^{n_r+j_1}+\cdots+p^ua_{n_r}\ep.
\end{equation}

Our goal is to show that $\alpha=0$. On the contrary, we assume that $1\le \alpha\le u/2$.

First we will show that $1\le \alpha\le u/2$ is not possible if $n_r\ge n_1+n_{r-1}=3+n_{r-1}.$ Let $n_r\ge 3+n_{r-1}.$  Then the term with  second largest exponent of $x$ in 
\begin{equation}\label{eq2}
f(x)\tilde{f}(x)=p^ua_{n_r}\ep x^{2n_r}+a_{n_r}a_{n_1} x^{2n_r-3}+p^u\ep a_{n_{r-1}}x^{n_r+n_{r-1}}+\cdots+p^ua_{n_r}\ep,
\end{equation}
is either $a_{n_r}a_{n_1}x^{2n_r-3}$ or $a_{n_r}a_{n_1}x^{2n_r-3}+p^u\ep a_{n_{r-1}} x^{n_r+n_{r-1}}$. Because of the condition   $p\nmid a_{n_1}a_{n_r}$ in the  definition of $\S_{n_1},$  the coefficient of the second largest exponent of $x$ in Equation~\eqref{eq2} is not divisible by $p.$ Therefore, if we are  able to show that the corresponding coefficient in Equation~\eqref{eq1} is always divisible by $p,$ then we arrive at a contradiction which in turn implies that the assumption $1\le \alpha\le u/2$ is not correct, and hence $\alpha$ has to be zero. 
So we aim to find out the coefficient of the second largest exponent of $x$ in Equation~\eqref{eq1} and will show that it is divisible by $p.$ That coefficient depends on $j_t$ and $j_1$ and the possible cases for $j_t$ and $j_1$ are as follows.
\begin{multicols}{3}
\begin{enumerate}
 \item  $j_t=3$ or $j_1=n_r-3$
 \item  $j_t>3$ and $j_1<n_r-3$
 \item  $j_t>3$ and $j_1>n_r-3$
 \item  $j_t<3$ and $j_1<n_r-3$
 \item  $j_t<3$ and $j_1>n_r-3.$
\end{enumerate}
\end{multicols}

 If $j_t=3$ or $j_1=n_r-3,$ then we are through as $p|b_0$ and $p|b_{n_r}$.  If $j_t>3, j_1<n_r-3$, then for every $i$
 $$2n_r-j_i\le 2n_r-j_t<2n_r-3,$$ and  for every $i\ne l$ $$n_r+j_i-j_l<n_r+j_i<2n_r-3.$$ Hence the second largest exponent in $g(x)\tilde{g}(x)$ is less than $2n_r-3$ implies that the case $j_t>3$ and $j_1<n_r-3$ cannot arise.

Let $j_t>3$ and $j_1>n_r-3.$ Then 
$$2n_r-j_i\le 2n_r-j_t<2n_r-3$$
for every $i$ and $j_1>n_r-3$ implies either $j_1=n_r-1$ or $j_1=n_r-2.$ If $j_1=n_r-1$, then $x^{2n_r-1}$ has coefficient $b_0b_{j_1}(\ne 0)$ in $g(x)\tilde{g}(x)$ while the term is absent in $f(x)\tilde{f}(x).$ Similar case arise when $j_1=n_r-2.$

With little work in the similar manner, one can show that the case $j_t<3$ and $j_1<n_r-3$ is also not possible. 

Let $j_t< 3$ and $j_1> n_r-3$. There are two possibilities: either $j_t=1, j_1>n_r-3$ or $j_t=2, j_1>n_r-3$. We consider both the cases separately. 

{\em Case I:} Let $j_t=1$ and $j_1>n_r-3$. If $j_1=n_r-2,$ then $x^{2n_r-1}$ has coefficient $b_{n_r}b_{j_t}$ in Equation \eqref{eq1} while $x^{2n_r-1}$ is absent in Equation \eqref{eq2}. So, $j_1$ has to be $n_r-1$ and $b_{n_r}b_{j_t}+b_0b_{j_1}=0$. By using the values of $b_0$ and $b_{n_r}$, we deduce that
\begin{equation}\label{neq1}
b_{j_1}=-\frac{p^{u-2\alpha}d_1b_{j_t}}{d},
\end{equation} 
and hence $p^{u-2\alpha}|b_{j_1}$. Similar to the values of $j_1, j_t,$ we now consider the different possible values of $j_2$ and $j_{t-1}$. Note that it is not possible to hold $j_{t-1}>3, j_2<n_r-3$ simultaneously. Otherwise $g(x)\tilde{g}(x)$ has second largest exponent $<2n_r-3$. 

Let $j_{t-1}=3$ or $j_2=n_r-3$. Then the coefficient of $x^{2n_r-3}$ in Equation \eqref{eq1} is 
\begin{equation*}
\begin{cases}
b_{n_r}b_{j_{t-1}}+b_{0}b_{j_2} &\mbox{ if $j_{t-1}=3, j_2=n_r-3$;}\\
b_{n_r}b_{j_{t-1}} &\mbox{ if $j_{t-1}=3, j_2\ne n_r-3$;}\\
b_0b_{j_2} &\mbox{ if $j_{t-1}\ne 3, j_2=n_r-3$,}
\end{cases}
\end{equation*}
each of them is divisible by $p$. 
 
Let $j_{t-1}< 3$ and $j_2<n_r-3$. Since $j_{t-1}=2$, the coefficient of $x^{2n_r-3}$ in Equation \eqref{eq1} is 
\begin{equation*}
\begin{cases}
b_{j_1}b_{j_{t-1}}+b_{n_r}b_{j_{t-2}} & \mbox{ if $j_{t-2}=3$;}\\
b_{j_1}b_{j_{t-1}} &\mbox{ otherwise.}
\end{cases}
\end{equation*}
From Equation \eqref{neq1}, $p$ will divide the above coefficient provided $u\ne 2\alpha$. 

If $u=2\alpha$, then Equation \eqref{neq1} reduces to
\begin{equation}\label{neq2}
b_{j_1}=-\frac{d_1b_{j_t}}{d}.
\end{equation}
Since $j_t=1, j_1=n_r-1, j_{t-1}=2, j_2<n_r-3$, the coefficient of $x^{2n_r-2}$ in $g(x)\tilde{g}(x)$ is $b_{j_1}b_{j_t}+b_{j_{t-1}}b_{n_r}=0$. As $p|b_{n_r},$ using Equation \eqref{neq2}, $p|b_{j_t}$, which in turn implies that $p|b_{j_1}$. Thus, if $u=2\alpha$, then also $p$ divides the coefficient of $x^{2n_r-3}$ in $g(x)\tilde{g}(x)$. 

Let $j_{t-1}>3$ and $j_2>n_r-3$. As $j_2=n_r-2$, the coefficient of $x^{2n_r-2}$ in $g(x)\tilde{g}(x)$ is $$b_0b_{j_2}+b_{j_1}b_{j_t}=0.$$ If $u\ne 3\alpha$, then using \eqref{neq1} in the last equation,  either $p|b_{j_2}$ or $p|b_{j_t}$. The coefficient of $x^{2n_r-3}$ in \eqref{eq1} is then
\begin{equation*}
\begin{cases}
b_{j_2}b_{j_t}+b_{j_3}b_0 & \mbox{ if $j_3=n_r-3;$}\\
b_{j_2}b_{j_t} &\mbox{ otherwise,}
\end{cases}
\end{equation*} 
each of which is divisible by $p$.

Let $j_{t-1}< 3$ and $j_2> n_r-3$. Then $j_{t-1}=2, j_2=n_r-2$ and the coefficient of $x^{2n_r-2}$ in \eqref{eq1} is 
\begin{equation}\label{chap6:eq5}
b_{n_r}b_{j_{t-1}}+b_{j_1}b_{j_t}+b_0b_{j_2}=0.
\end{equation}
On the other hand, the coefficient of $x^{2n_r-3}$ in \eqref{eq1} is 
\begin{equation*}
\begin{cases}
b_{j_1}b_{j_{t-1}}+b_{j_2}b_{j_t}+b_{n_r}b_{j_{t-2}}+b_0b_{j_3} &\mbox{ if $j_{t-2}= 3, j_3=n_r-3$;}\\ 
b_{j_1}b_{j_{t-1}}+b_{j_2}b_{j_t}+b_{n_r}b_{j_{t-2}} &\mbox{ if $j_{t-2}=3, j_3\ne n_r-3$;}\\
b_{j_1}b_{j_{t-1}}+b_{j_2}b_{j_t}+b_0b_{j_3} &\mbox{ if $j_{t-2}\ne 3, j_3=n_r-3$;}\\
b_{j_1}b_{j_{t-1}}+b_{j_2}b_{j_t} &\mbox{ if $j_{t-2}\ne 3, j_3\ne n_r-3$.}
\end{cases}
\end{equation*}
Let $u=2\alpha$. By using \eqref{neq2} and \eqref{chap6:eq5}, we get
\begin{equation*}
p^{\alpha}d_1b_{j_{t-1}}-\frac{d_1b_{j_t}^2}{d}+p^{\alpha}db_{j_2}=0.
\end{equation*}
From the last equation, $p|b_{j_t}$ and hence $p|b_{j_1}$ by \eqref{neq2}. This implies that the coefficient of $x^{2n_r-3}$ in \eqref{eq1} is divisible by $p$. 

Let $u>2\alpha$ and $u\ne 3\alpha$. By using \eqref{neq1}, Equation \eqref{chap6:eq5} reduces to 
\begin{equation*}
p^{u-\alpha}d_1b_{j_{t-1}}-\frac{p^{u-2\alpha}d_1b_{j_t}}{d}+p^{\alpha}db_{j_2}=0.
\end{equation*}
If $u<3\alpha,$ then $p$ would divide $b_{j_t}$ and $p$ already divides $b_{j_1}$ by \eqref{neq1}. If $u>3\alpha$, then $p|b_{j_2}$. Hence, in this particular case also, the coefficient of $x^{2n_r-3}$ is divisible by $p$ in \eqref{eq1}.

{\em Case II:} Let $j_t=2$ and $j_1>n_r-3$. With a  similar analysis, it can be seen that either $j_{t-1}=3$ or $j_2=n_r-3$ or both has to be true. But in those cases, the corresponding coefficient is divisible by $p$ in $g(x)\tilde{g}(x)$. 

If $n_r=n_{r-1}+1$ or $n_r=n_{r-1}+2$, then instead of considering the second largest exponent in \eqref{eq1} and \eqref{eq2}, we will consider the coefficients of $x^{2n_r-3}$ in both the equations. With a similar analysis, one can show that the coefficient of $x^{2n_r-3}$ in Equation \eqref{eq1} is divisible by $p$ while it is not the case in Equation \eqref{eq2}. Therefore, $\alpha$ has to be zero. 
\end{proof}

The lemma is even true for polynomials belonging to $\S_1\cup \S_2'.$ In other words, 
 
\begin{theorem}\label{mainlem}
Let $f(x)=a_{n_r}x^{n_r}+a_{n_{r-1}}x^{n_{r-1}}+\cdots+a_{n_1}x^{n_1}+p^u\ep\in \S_1\cup \S_2'\cup \S_3'$ be reducible. Then the constant term of the one of the irreducible factors of $f(x)$ is $|f(0)|$. 
\end{theorem} 

\begin{proof}
We use  the  same notations as  used in the proof of  Lemma \ref{lem1}. We have
\begin{equation}\label{eq4}
f(x)\tilde{f}(x)=p^ua_{n_r}\ep x^{2n_r}+a_{n_r}a_{n_1} x^{2n_r-n_1}+p^u\ep a_{n_{r-1}}x^{n_r+n_{r-1}}+\cdots+p^ua_{n_r}\ep,
\end{equation}
and 
\begin{equation}\label{eq5}
g(x)\tilde{g}(x)=p^ua_{n_r}\ep x^{2n_r}+b_{n_r}b_{j_t}x^{2n_r-j_t}+b_0b_{j_1}x^{n_r+j_1}+\cdots+p^ua_{n_r}\ep.
\end{equation}

It is sufficient to consider $n_1=1,2$. If $n_1=1$, then either $j_t=1$ or $j_1=n_r-1.$  The coefficient of $x^{2n_r-1}$ is then divisible by $p$ in \eqref{eq5} but not in \eqref{eq4}. 

Suppose $n_2=1$ and $n_r\ge 2+n_{r-1}$. If $j_t=2$ or $j_1=n_r-2,$  then the coefficient of $x^{2n_r-2}$ is divisible by $p$ in \eqref{eq5} but not in \eqref{eq4}. Since the term $x^{2n_r-1}$ is absent in \eqref{eq4}, we cannot have $j_t=1, j_1<n_r-1$ or $j_t>1, j_1=n_r-1$. Hence, $j_t=1, j_1=n_r-1$ and $b_{n_r}b_{j_t}+b_0b_{j_1}=0$. Since $u$ is odd, this would imply $p|b_{j_1}.$ Also $j_t<j_{t-1}$ and $j_2<j_1$ implies that either $j_{t-1}=2$ or $j_2=n_r-2.$ Then the coefficient of $x^{2n_r-2}$ in equation~\eqref{eq5} is
\begin{equation*}
\begin{cases}
b_{j_1}b_{j_t}+b_0b_{j_2}+b_{n_r}b_{j_{t-1}} & \mbox{ if $j_{t-1}=2, j_2=n_r-2$;}\\
b_{n_r}b_{j_{t-1}}+b_{j_1}b_{j_t} & \mbox{ if $j_{t-1}=2, j_2\ne n_r-2$;}\\
b_{j_1}b_{j_t}+b_0b_{j_2} &\mbox{ if $j_{t-1}\ne 2, j_2=n_r-2$;}\\
b_{j_1}b_{j_{t}} &\mbox{ if $j_{t-1}\ne 2, j_2\ne n_r-2$,}
\end{cases}
\end{equation*}
each divisible by $p$. Similarly if $n_r=n_{r-1}+1,$ then one can arrive at the same kind of contradiction by comparing the coefficient of $x^{2n_r-2}$ in Equation \eqref{eq4} and \eqref{eq5}. Hence $\alpha=0.$
\end{proof}

\begin{proof} [\unskip\nopunct]{\textbf{Proof of Theorem~\ref{mainthm}:}} Follows from Theorem~\ref{mainlem} and Remark~\ref{rem:non-reciprocal}.
\end{proof}

\begin{proof} [\unskip\nopunct]{\textbf{Proof of Corollary~\ref{cor:n_1=2}:}}
Let $f(x)\in \S_2$. From the proof of Theorem \ref{mainlem}, if $f(x)$ is reducible then $\alpha$ has to be either $u/2$ or $0$. If $f(x)\in \S_3$ is reducible, then from the proof of Lemma \ref{lem1}, $\alpha$ has to be either $u/3$ or $0$. Because of the hypothesis, in either case,  $\alpha$ can't be $0$, from which the result follows. 
\end{proof}

\begin{proof} [\unskip\nopunct]{\textbf{Proof of Theorem~\ref{cyclothm}:}}
Suppose $f(x)=f_1(x)f_2(x)$ is a proper factorization of $f(x)$.  By using Theorem~\ref{mainlem}, without loss of generality, we can  assume that  $|f_1(0)|=1$, $|f_2(0)|=p^u.$ As a consequence,  $f_2(x)$ is irreducible.

With a proof similar to that of Remark~\ref{rem:non-reciprocal} one can prove that  all the roots of $f(x)$ lies in the region $|z|\ge 1$. Let $z_1, z_2, \ldots, z_s$ be all the roots of $f_1(x)$, where $\deg(f_1)=s<\deg(f).$ Then  
\begin{equation*}
\prod\limits_{i=1}^s |z_i|=\frac{1}{|d|},
\end{equation*}
where $d$ is the leading coefficient of $f_1(x)$, dividing $a_{n_r}$. 
Since $z_i$'s are roots of $f(x)$, we have $|z_i|\ge 1$ and hence $|d|=1$.  Consequently, all the roots of $f_1(x)$ lies on the unit circle and by Kronecker's theorem $\pm f_1(x)$ becomes a product of cyclotomic polynomials.  

The second part is a special case of the proof of Proposition~\ref{pro:fc(x)}.
\end{proof}

\begin{proof}[\unskip\nopunct]{\textbf{Proof of Proposition~\ref{pro:fc(x)}:}}
 Let $\zeta$ be a primitive $t^{\text{th}}$ root of unity with $f(\zeta)=0$. Then 
\begin{equation}\label{eq7}
-a_0=a_{n_1}\zeta^{n_1}+a_{n_2}\zeta^{n_2}+\cdots+a_{n_r}\zeta^{n_r}.
\end{equation} 
Taking modulus on both sides
\begin{equation*}
|a_0|=|a_{n_1}\zeta^{n_1}+a_{n_2}\zeta^{n_2}+\cdots+a_{n_r}\zeta^{n_r}|=\sum\limits_{i=1}^r|a_{n_i}|.
\end{equation*}
From triangle inequality, the last two equations hold if and only if the ratio of any two parts is a positive real number. Therefore, $a_{n_r}\zeta^{n_r-n_i}/a_{n_i}=|a_{n_r}\zeta^{n_r-n_i}/a_{n_i}|$ gives $\zeta^{n_r-n_i}=\sgn(a_{n_r}a_{n_i})$ for $1\le i\le r-1$. From \eqref{eq7}, we have
\begin{equation*}
-a_0= a_{n_i}\zeta^{n_i}\left[ \left|\frac{a_{n_1}}{a_{n_i}}\right|+\dots+\left|\frac{a_{n_{i-1}}}{a_{n_i}}\right|+1+\left|\frac{a_{n_{i+1}}}{a_{n_i}}\right|+\dots+\left|\frac{a_{n_r}}{a_{n_i}}\right|\right],
\end{equation*}
so that $\zeta^{n_i}=-\sgn(a_0a_{n_i})$. From $\zeta^{n_r-n_i}\zeta^{n_i}$, one gets the last equation. Remaining all the equations satisfied by $\zeta$ can be drawn from these $r$ equations. Conversely, if $\zeta$ satisfy each of the $r$ equations $x^{n_i}+\sgn(a_0 a_{n_i})=0,$ then $f(\zeta)=0$. It remains to show the separability of the cyclotomic part of $f(x)$.  Let $\zeta$ be a roots of unity satisfying $x^{n_i}+\sgn(a_0 a_{n_i})=0$ for $1\le i\le r$ and $f(\zeta)=0, f'(\zeta)=0$. Using the $r$ relations satisfied by $\zeta$ in $f'(\zeta)=0$, we derive that 
\begin{equation*}
n_r|a_{n_r}|+\cdots+|a_{n_1}|n_1=0,
\end{equation*} 
which is not possible. 
\end{proof}

\begin{proof} [\unskip\nopunct]{\textbf{Proof of Corollary~\ref{cor:pos}:}}
For any $g(x)\in \Z[x]$, $g(x)$ has cyclotomic factor if and only if $g(x^d)$ has a cyclotomic factor, $d\ge 1$. Therefore it is sufficient to prove  the result for polynomials whose exponents are relatively prime. The result follows from Theorem \ref{cyclothm}, and Lemma \ref{lem:basic_cyclo}.
\end{proof}
We now see an application of Corollary~\ref{cor:pos}.  In~\cite{bksr} we showed that, if $(n_1,n_2,\ldots, n_p)=1$ then   $x^{n_1}+x^{n_2}+\cdots+x^{n_p}+p$ is irreducible. The same result  is true when the prime number is replaced with a prime power and the power is not divisible by $2$ and $3$. In particular, if $p$ is a prime number and $n_1>n_2>\ldots> n_p$ are positive integers with $(n_1,n_2, \ldots, n_p)=1, n_p\le 3$, then $x^{n_1}+x^{n_2}+\cdots+x^{n_p}+p^u$ is irreducible for any integer $u$ with  $(u,6)=1$. 

\begin{cor}
Let $f(x)\in \S_1\cup \S_2'\cup \S_3'$ be a polynomial with $n_{j-1}=n_j-1$ for some $j$, $2\le j\le r$ and $|a_{n_r}|+|a_{n_r-1}|+\cdots+|a_{n_1}|=p^u$. Then $f(x)$ is reducible if and only if either $f(1)=0$ or $f(-1)=0$. 
\end{cor}

From Lemma \ref{lem:basic_cyclo}, $(x^n\pm 1 , x^{n-1}\pm 1)$ is either $1$ or $x\pm 1$. Hence the proof of the above Corollary  follows directly by applying  Theorem~\ref{cyclothm}.

\section{Applications}\label{sec:app}
Suppose $u\ge 2$ and $a,b,p\in \N, p$ being a prime number, $p\nmid ab$. In this section we consider the trinomials of the form $f(x)=ax^n+b\ep_1x^m+p^u\ep_2,$ where $\ep_i\in \{\, -1,+1\,\}$ and $n>m>0$. One can see that results in the previous section are applicable for trinomials. In this section, we will discuss the reducibility of $f$ in the case $p^u=a+b$. From above, we know  
\begin{theorem}
 Let $f(x)=ax^{n}+b\ep_1x^{m}+p^u\ep_2\in \S_1\cup \S_2'\cup \S_3'$  and $p^u=a+b$. Then apart from cyclotomic factors, $f(x)$ has only  one irreducible non-reciprocal  polynomial.  
\end{theorem}
From Theorem \ref{cyclothm}, all such $f(x)$ are separable over $\Q$. Here $m\le 3$. However, one can find the separability criterion for a bigger class of  trinomials with arbitrary values of $m$ by using the discriminant formula.  

\begin{theorem}[C.R. Greenfield, D. Drucker~\cite{GD}]\label{th0}
The discriminant of the trinomial $x^n+ax^m+b$ is 
\begin{equation*}
D=(-1)^{\binom{n}{2}}b^{m-1}\left[ n^{n/d}b^{n-m/d}-(-1)^{n/d}(n-m)^{n-m/d}m^{m/d}a^{n/d}\right]^d,
\end{equation*}
where $d=(n,m),$ and $a,b\in \Z\setminus \{\, 0\,\}$.
\end{theorem}

 Note that  if $h(x)\in \Z[x],h(0)\ne 0,$ then $h(x)$ is separable if and only if $h(x^k)$ is separable for every $k\ge 1.$ Hence in order to 
check separability of polynomials whose constant term is nonzero, it is sufficient to consider the polynomials whose gcd of the exponents is $1.$

\begin{theorem}\label{sepa}
Let $a,b,p\in \N$, $p$ be a prime number, $p\nmid ab$ and $(a,b)=1$,  $(n,m)=1$. Then $f(x)=ax^{n}+b\ep_1x^{m}+p^u\epsilon_2,$ where $u\ge 2, b<p^u$  and $\epsilon_i\in\{-1,1\}$ is not  separable over $\Q$ if and only if  $b=n, p|m,  p^{u(n-m)}a^m=(n-m)^{n-m}m^m, \ep_2^{n-m}(-\ep_1)^n=1$. 
\end{theorem}

\begin{proof} From Theorem \ref{th0}, the discriminant of $f$ is 
\begin{equation*}
D_f=(-1)^{\binom{n}{2}}(p^u\epsilon_2)^{n-m}a^{n-m-1}\left[ n^{n/d}(p^u\epsilon_2)^{n-m/d}a^{m/d}-(-1)^{n/d}(n-m)^{n-m/d}m^{m/d}(b\epsilon_1)^{n/d}\right]^d,
\end{equation*}
where $d=(n,m)$. Here $d=1.$
Now  $f(x)$ has a multiple root if and only if $D_f=0$, {\it i.e.,}
\begin{equation}\label{eq:separable}
n^{n}(p^u\epsilon_2)^{n-m}a^m=(-1)^{n}(n-m)^{n-m}m^{m}(b\epsilon_1)^{n}.
\end{equation}

Since $(n,m)=(n,n-m)=1$ and $n^{n}|(-1)^{n}(n-m)^{n-m}m^{m}(b\epsilon_1)^{n}$, we have $n|b.$
Let $b=ns$ for some $s\in \N.$  Equation~(\ref{eq:separable}) then becomes 
$$(p^u\epsilon_2)^{n-m}a^m=(-1)^{n}(n-m)^{n-m}m^{m}(\epsilon_1)^{n}s^n.$$ Thus we have $s^n|(p^u\epsilon_2)^{n-m}a^m$ and from  the hypothesis we have $(p,s)=(a,s)=1.$ Hence $s=1$ and $b=n.$  The last equation reduces to 
\begin{equation*}
p^{u(n-m)}\ep_2^{n-m}a^m=(-\ep_1)^n(n-m)^{n-m}m^m.
\end{equation*}

As a consequence, either  $p|m$ or $p|n-m.$ In order to show that $p|m,$ it is sufficient to show $p\nmid n-m.$ If $n-m=1$, then $p\nmid n-m.$ Suppose $n-m>1$ and $p|n-m.$ From above equation we have $p^u|n-m.$ Consequently, $p^u\le n-m<n=b<p^u$ is a contradiction. Finally the equations
\begin{equation*}
p^{u(n-m)}a^m=(n-m)^{n-m}m^m,
\end{equation*}
and $\ep_2^{n-m}(-\ep_1)^n=1$ 
 follow easily from the last equation. Converse part is clear. But for the converse we do not require $p|m.$
\end{proof}

The following  example illustrates  all the  conditions given in Theorem~\ref{sepa}. Let $p$ be an odd prime. Then we have  
\begin{equation*}
x^{p+1}+(p+1)x^p+p^p=(x+p)^2g(x),\mbox{where}\;\;g(x)\in \Z[x].
\end{equation*}

\begin{cor}
Let $f(x)=ax^{n}+b\ep_1x^{m}+p^u\ep_2\in (\S_2\setminus \S_2')\cup (\S_3\setminus \S_3')$, where $\epsilon_i\in\{-1,1\}$, $(a,b)=(m,n)=1$  and $b<p^u.$ Then $f(x)$ is separable over $\Q$ except for  $x^3+3\ep_1x^2-4\ep_1$ and $x^4+4\ep_1x^3+27$. 
\end{cor}
\begin{proof} \begin{description}
               \item[m=2:] Let $f\in \S_2\setminus \S_2'$ and satisfying the conditions of hypothesis. If $f$ is not separable, then from  Theorem~\ref{sepa} $p|m$ implies $p=2$ and $n$ is odd. Further we have $2^{u(n-2)}a^2=(n-2)^{n-2}4.$ It is easy to see that if $u\ne 2$ or $n\ne 3,$ then we get a contradiction. If $u=2$ and $n=3,$ then $f(x)=x^3+3\ep_1x^2-4\ep_1=(x-\ep_1)(x+2\ep_1)^2$ follows from $\ep_2\ep_1=-1$. 
               \item[m=3:] Let $f\in \S_3\setminus \S_3'$ and satisfying the conditions of hypothesis. If $f$ is not separable, then from  Theorem~\ref{sepa} $p|m$ implies $p=3$ and  $3\nmid n-3$. If $n\ge 5,$ then $u$ being greater than 1, $3|n-3$ is a contradiction. If $n=4,$ then $u=3$ and $a^3=1$. In other words, $f(x)=x^4+4\ep_1x^3+27\ep_2$. From $\ep_2(-\ep_1)^4=1$, we get $\ep_2=1$ and  $x^4+4\ep_1x^3+27=(x+3\ep_1)^2(x^2-2\ep_1x+3)$.
              \end{description}

\end{proof}

Since trinomials in $\S_1\cup \S_2'\cup \S_3'$ have simple zeros, $\Phi_t(x)^2\nmid f(x)$ for any $t\ge 1$. With this observation, we will characterize the irreducibility criteria for trinomials. One can recall Proposition~\ref{pro:fc(x)} to know the cyclotomic factors of $ax^n+b\ep_1 x^m+\ep_2 p^u,$ where $p^u=a+b$ and $u\ge 2$ for any $m\ge 1.$

\begin{theorem}\label{thm:redoftri}
Let $p$ be a prime, $a,b\in \N$ and  $u\ge 2, p^u=a+b$. Let $f(x)=ax^{n}+b\ep_1x^{m}+p^u\ep_2\in \S_1\cup \S_2'\cup \S_3'$ be a polynomial of degree $n$.   
\begin{enumerate}
\item \label{thm:redoftri:1}  If $\ep_1=1$, then $f(x)$ is reducible if and only if $e(n)=e(m)$ and in that case $f_c(x)=x^{(n,m)}+\sgn(\ep_2).$
\item \label{thm:redoftri:2} If $\ep_1=-1,$ then  $f(x)$ is reducible if and only if  $\ep_2e(m)>\ep_2e(n).$ Moreover the reciprocal part of $f(x)$ is  $f_c(x)=x^{(n,m)}+1$.
\end{enumerate}
\end{theorem}

\begin{proof} We prove the result only for the case $\ep_1=\ep_2=-1.$ One can prove the remaining three  cases in similar lines.
In this case, we have to show  that $f(x)$ is reducible if and only if $e(m)<e(n).$

From Theorem~\ref{cyclothm}, $f(x)$ is irreducible if and only if $f_c(x)=1$, where   $f_c(x)$ is the greatest common divisor of $x^n-1$ and $x^m+1$. From Lemma~\ref{lem:basic_cyclo}, we have \begin{equation*}
f_c(x)=(x^{n}-1, x^{m}+1)=\begin{cases} 
x^{(n/2,m)}+1 & \mbox{ if $e(n)\ge 2e(m);$}\\
1 & \mbox{ otherwise.}
\end{cases}
\end{equation*}

Thus $f(x)$ reducible if and only if $e(n)\ge 2e(m)$, that is, if and only if $e(m)<e(n).$ 

Suppose $d=\gcd(m,n)$. Then there exists 
$n_1,m_1\in \N$ such that $n=dn_1, m=dm_1$ and 
$(n_1,m_1)=1.$ If $\ze$ denotes a primitive $2d^{th}$ root of unity, then
$$a \ze^{dn_1}-b\ze^{dm_1}-p^u=a+b-p^u=0,$$
as $\ep_1=\ep_2=-1$, $n_1$ even ($e(n)>e(m)$) and $p^u=a+b.$ 
\end{proof}

We conclude the paper with few comments on generalization of Theorem \ref{mainthm}. It is natural to ask whether Theorem \ref{mainthm} can be extended to arbitrary $n_1$. The example 
\begin{equation*}
x^8+x^6+x^4+4=(x^4-x^3+x^2-2x+2)(x^4+x^3+x^2+2x+2),
\end{equation*}
given by Jankauskas\cite{jankauskas} suggests  that the
 generalization is not possible as $x^8+x^6+x^4+4\in\S_4'$ and is reducible.

The example $x^{p+1}+(p+1)x^p+p^p$  shows that Theorem \ref{mainthm} is not true for $f\in\S_p\backslash \S_p'$, $p$ being odd prime. We conjecture that 
\begin{conj}
Let $p$ and $q$ be two prime numbers, $u\ge 2$ and $\ep\in \{-1,1\}$. 
Suppose $f(x)=a_{n_r}x^{n_r}+\cdots+a_{q}x^{q}+p^u\ep\in \S_q'.$ If $p^u>|a_q|+\cdots+|a_{n_r}|,$ then $f(x)$ is irreducible. 
\end{conj}

\end{document}